\documentclass[11pt]{article}

\usepackage[utf8]{inputenc}
\usepackage[margin=1.2in]{geometry}
\usepackage{amsmath,amsthm,amssymb,amsfonts,csquotes}
\usepackage{mathrsfs}
\usepackage{graphicx}
\usepackage[colorlinks=true, linkcolor=blue, citecolor=blue, urlcolor=blue, breaklinks=true]{hyperref}
\usepackage[capitalise]{cleveref}
\usepackage{tikz,float}
\usepackage{subcaption}
\usepackage{xcolor}
\usetikzlibrary{decorations.pathreplacing}

\newtheorem{theorem}{Theorem}[section]

\newtheorem{conjecture}[theorem]{Conjecture}
\newtheorem{corollary}[theorem] {Corollary}
\newtheorem{definition}[theorem]{Definition}

\newtheorem{lemma}[theorem]{Lemma}

\newtheorem{proposition}[theorem]{Proposition}

\newtheorem{result}[theorem]{Result}

\title{Exploring the Influence of Graph Operations\\on Zero Forcing Sets}

\author{Krishna Menon\footnote{Department of Mathematics, Chennai Mathematical Institute, India. Email: krishnamenon@cmi.ac.in} \ and Anurag Singh\footnote{Department of Mathematics, Indian Institute of Technology Bhilai, India. Email: anurags@iitbhilai.ac.in}}

\date{}

\begin{document}

\maketitle

\begin{abstract}
    Zero forcing in graphs is a coloring process where a colored vertex can \emph{force} its unique uncolored neighbor to be colored. 
    A zero forcing set is a set of initially colored vertices capable of eventually coloring all vertices of the graph. 
    In this paper, we focus on the numbers $z(G; i)$, which is the number of zero forcing sets of size $i$ of the graph $G$. 
    These numbers were initially studied by Boyer et al. where they conjectured that for any graph $G$ on $n$ vertices, $z(G; i) \leq z(P_n; i)$ for all $i \geq 1$ where $P_n$ is the path graph on $n$ vertices. 
    The main aim of this paper is to show that several classes of graphs, including outerplanar graphs and threshold graphs, satisfy this conjecture. 
    We do this by studying various graph operations and examining how they affect the number of zero forcing sets.
\end{abstract}

\noindent\textbf{Keywords:} Zero forcing set, path graph, tree, outerplanar graph, threshold graph\\
\textbf{MSC2020:} 05C05, 05C15, 05C30, 05C69

\section{Introduction}
Suppose we are given a simple graph $G$ where some vertices are colored. 
The \emph{zero forcing rule} is the following: 
If there is a colored vertex $v$ such that its neighborhood contains a unique uncolored vertex $w$, then $w$ can be colored. 
A \emph{zero forcing set} of $G$ is a set of vertices $S$ such that if we start with the vertices in $S$ colored, we can eventually color all vertices of $G$ by repeatedly applying the zero forcing rule. 
We sometimes simply say the set $S$ is forcing in $G$. 
If a vertex $v$ is used to color a vertex $w$, we say that $v$ \emph{forces} $w$.

\begin{figure}[H]
    \centering
    \begin{tikzpicture}[scale = 0.8]
        \node [circle, draw = black, fill = cyan] (a) at (0, 2) {};
        \node [circle, draw = black] (b) at (1, 3) {};
        \node [circle, draw = black] (c1) at (2, 2) {};
        \node [circle, draw = black, fill = cyan] (d) at (1, 1) {};

        \node [ultra thick, circle, draw = red!60, fill = cyan] (c2) at (3, 3) {};
        \node [circle, draw = black] (g) at (3, 1) {};

        \draw (a)--(b)--(c1)--(c2);
        \draw (g)--(c1)--(d)--(a);
        \draw (d) -- (b);

        \node at (4, 2) {$\rightarrow$};
    \end{tikzpicture}
    \hspace{0.1cm}
    \begin{tikzpicture}[scale = 0.8]
        \node [ultra thick, circle, draw = red!60, fill = cyan] (a) at (0, 2) {};
        \node [circle, draw = black] (b) at (1, 3) {};
        \node [circle, draw = black, fill = cyan] (c1) at (2, 2) {};
        \node [circle, draw = black, fill = cyan] (d) at (1, 1) {};

        \node [circle, draw = black, fill = cyan] (c2) at (3, 3) {};
        \node [circle, draw = black] (g) at (3, 1) {};

        \draw (a)--(b)--(c1)--(c2);
        \draw (g)--(c1)--(d)--(a);
        \draw (d) -- (b);

        \node at (4, 2) {$\rightarrow$};
    \end{tikzpicture}
    \hspace{0.1cm}
    \begin{tikzpicture}[scale = 0.8]
        \node [circle, draw = black, fill = cyan] (a) at (0, 2) {};
        \node [circle, draw = black, fill = cyan] (b) at (1, 3) {};
        \node [ultra thick, circle, draw = red!60, fill = cyan] (c1) at (2, 2) {};
        \node [circle, draw = black, fill = cyan] (d) at (1, 1) {};

        \node [circle, draw = black, fill = cyan] (c2) at (3, 3) {};
        \node [circle, draw = black] (g) at (3, 1) {};

        \draw (a)--(b)--(c1)--(c2);
        \draw (g)--(c1)--(d)--(a);
        \draw (d) -- (b);

        \node at (4, 2) {$\rightarrow$};
    \end{tikzpicture}
    \hspace{0.1cm}
    \begin{tikzpicture}[scale = 0.8]
        \node [circle, draw = black, fill = cyan] (a) at (0, 2) {};
        \node [circle, draw = black, fill = cyan] (b) at (1, 3) {};
        \node [circle, draw = black, fill = cyan] (c1) at (2, 2) {};
        \node [circle, draw = black, fill = cyan] (d) at (1, 1) {};

        \node [circle, draw = black, fill = cyan] (c2) at (3, 3) {};
        \node [circle, draw = black, fill = cyan] (g) at (3, 1) {};

        \draw (a)--(b)--(c1)--(c2);
        \draw (g)--(c1)--(d)--(a);
        \draw (d) -- (b);
    \end{tikzpicture}
    \caption{Zero forcing in a graph with the forcing vertex highlighted red at each step.}
\end{figure}
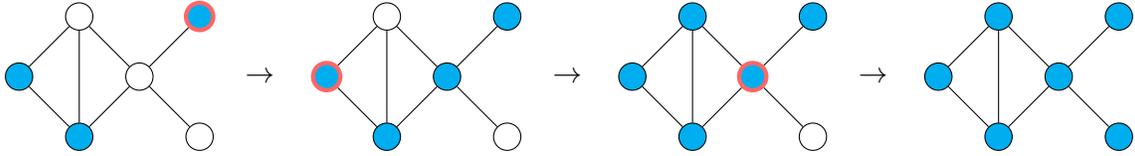

For a graph $G$ and $i \geq 1$, we set $Z(G; i)$ to be the collection of forcing sets of $G$ of cardinality $i$. 
Similarly, $Z'(G; i)$ is the set of non-forcing sets of $G$ of cardinality $i$. 
We also set $z(G; i) = \# Z(G; i)$ and $z'(G; i) = \# Z'(G; i)$. 
Note that $z(G; i) + z'(G; i)= \binom{n}{i}$ where $n$ is the number of vertices of $G$. 
The {\it zero forcing number} of $G$, denoted $z(G)$, is the minimum cardinality of a forcing set of $G$, i.e., $z(G) = \operatorname{min}\{i \mid z(G; i) \neq 0\}$.

\begin{figure}[H]
    \centering
    \begin{tikzpicture}
        \node [circle, draw = black, fill = cyan] (1) at (1, 0) {};
        \node [circle, draw = black, fill = cyan] (2) at (2, 0) {};
        \node [circle, draw = black] (3) at (3, 0) {};
        \node [circle, draw = black] (4) at (4, 0) {};
        \node [circle, draw = black, fill = cyan] (5) at (5, 0) {};
        \node [circle, draw = black] (6) at (6, 0) {};

        \node [circle, draw = black] (a) at (2.55, 0.75) {};
        \node [circle, draw = black, fill = cyan] (b) at (3.45, 0.75) {};

        \draw (1)--(2)--(3)--(4)--(5)--(6);
        \draw (a)--(3)--(b);
    \end{tikzpicture}
    \caption{A non-forcing set in a graph.}
\end{figure}
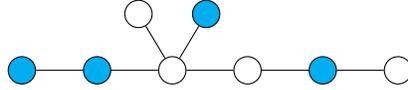

The study of zero forcing sets in graphs has garnered significant interest due to its role as a coloring process on graphs and its intriguing connections to various fields such as linear algebra \cite{ borio10, AIM, huang10}, power grid domination strategies \cite{powerdomi1, powerdomi2, powerdomi3}, theoretical computer science \cite{algo, yang13}, modeling various physical phenomena \cite{burg13, kforcing, dreyer09}, control of quantum systems \cite{quantum}, and rumor spreading models \cite{prob1}. Initially introduced by Burgarth and Giovannetti \cite{quantum}, the zero forcing process has since been examined extensively by researchers. The concept of zero forcing number, introduced by an AIM research group, serves as a bound for the maximum nullity of a graph \cite{AIM} and has connections with other graph parameters as well (see \cite{gc1, gc2, row11}).  The problem of zero forcing has also been studied in randomized setups \cite{sam24, prob1, prob2} and from a reconfiguration perspective \cite{geneson}.  A few other variants of zero forcing have also been well explored (see \cite{var1, var3, var4, var2,  var5}).

The counting problem associated with zero forcing sets was first explored by Boyer et al. in \cite{boyer19}. Here, the authors studied the numbers $z(G; k)$ for various classes of graphs, such as complete graphs, cycle graphs, threshold graphs, and more. 
However, the case of path graphs is of particular interest due to the following conjecture, which asserts that path graphs \emph{dominate} all other graphs in terms of number of forcing sets.

\begin{conjecture}{\cite[Question 2]{boyer19}}\label{conj}
    For any graph $G$ on $n$ vertices, $z(G; i) \leq z(P_n, i)$ for all $i \geq 1$ where $P_n$ is the path graph on $n$ vertices.
\end{conjecture}

There are relatively few results about graphs that satisfy this conjecture. 
It is known that graphs that have a Hamiltonian path or a fort of appropriate size (see \Cref{fort_def}) satisfy this conjecture \cite[Propositions 17, 18]{boyer19}. 
It is also known that graphs with a certain bound on the minimum degree of the vertices satisfy the conjecture (see \cite[Corollary 1.10]{sam24}). 
One of the principle aims of this paper is to expand the class of graphs known to satisfy \Cref{conj}.

We prove our results by studying various graph operations (such as deleting certain vertices or edges) and examining their effect on forcing sets. 
Using these operations, we obtain bounds on the numbers $z(G; i)$, which we then use to show that the graph $G$ satisfies \Cref{conj}.

In particular, we show that trees satisfy \Cref{conj}, which was a problem suggested by Sam Spiro \cite[Conjecture 3.7]{p2solve}. 
We in fact prove the following stronger result.

\begin{theorem}
    All outerplanar graphs (see \Cref{OP_def}) as well as all threshold graphs (see \Cref{thresh_def}) satisfy \Cref{conj}.
\end{theorem}

We also show that wedges of such graphs satisfy \Cref{conj} (see \Cref{TOPW}).

\section{Preliminaries}

%We use some standard terms from graphs theory in the sequel, which we define here for convenience. 
A \emph{graph} $G$ is an ordered pair of sets $(V(G), E(G))$ where $V(G)$ is called the set of \emph{vertices} and $E(G)$ is the set of \emph{edges}. 
Each element of $E(G)$ is a subset of $V(G)$ of size $2$. 
We say that two vertices $v, w$ are \emph{adjacent} if $\{v, w\} \in E(G)$. 
For any vertex $v \in V(G)$, we define its \emph{neighborhood} to be the set $N(v) := \{w \in V(G) \mid \{v, w\} \in E(G)\}$. 
The number of vertices adjacent to $v$ (i.e., $\# N(v)$) is called the \emph{degree} of $v$.

For any positive integer $n$, we use $[n]$ to denote the set $\{1, 2, \ldots, n\}$. 
Set $P_n$ to be the path graph on $n$ vertices, i.e., $V(P_n) = [n]$ and $E(P_n) = \{\{k, k + 1\} \mid k \in [n - 1]\}$. 
The \emph{cycle graph} on $n$ vertices $C_n$ is the graph obtained from $P_n$ by adding the edge $\{1, n\}$.

A \emph{subgraph} $H$ of a graph $G$ is a graph such that $V(H) \subseteq V(G)$ and $E(H) \subseteq E(G)$. 
We sometimes simply say that $G$ \emph{contains} $H$. 
For any set of vertices $W \subseteq V(G)$, we denote by $G \setminus W$ the subgraph of $G$ with vertex set $V(G) \setminus W$ and edge set $\{\{v, w\} \in E(G) \mid v, w \notin W\}$. 
We say that we have \emph{deleted} the vertices in $W$. 
A graph is said to be \emph{connected} if for any two vertices $v, w$, the graph contains a path whose endpoints are $v$ and $w$. 
The \emph{connected components} of a graph are its maximal connected subgraphs.

We now recall a few results from \cite{boyer19}. 
First, we describe the forcing sets in path graphs.

\begin{result}{\cite[Proposition 5]{boyer19}}\label{path_result}
    For any $n \geq 1$, we have $S \subseteq [n]$ is non-forcing in $P_n$ if and only if $1, n \notin S$ and $S$ does not contain consecutive numbers. 
    Hence, for any $i \geq 1$, $$z'(P_n; i) = \binom{n - i - 1}{i} \text{ and } z(P_n; i) = \binom{n}{i} - \binom{n - i - 1}{i}.$$
\end{result}

Before moving ahead, we recall the following notion, which helps in identifying certain graphs satisfying \Cref{conj}.

\begin{definition}\label{fort_def}
    A \emph{fort} in a graph $G$ is a non-empty set of vertices $F$ such that no vertex outside $F$ is adjacent to exactly one vertex in $F$. 
    That is, for any vertex $v \notin F$, we have $\#(N(v) \cap F) \neq 1$.
\end{definition}

Note that if $S \subset V(G)$ is such that $S \cap F = \phi$ for a fort $F$, then $S$ is a non-forcing set of $G$. 
This is because no vertex in $F$ can be colored. 
Using this fact, one can prove the following result.

\begin{result}{\cite[Proposition 17]{boyer19}}\label{fort}
    If a graph $G$ has a fort $F$ such that $\# F \leq z(G) + 1$, then $G$ satisfies \Cref{conj}.
\end{result}

We also have the following result for the \emph{union} of two graphs, which can be derived from \cite[Proposition 13]{boyer19}. 
For two graphs $G$ and $H$ (with disjoint vertex sets), we define their union $G \sqcup H$ to be the graph with vertex set $V(G) \cup V(H)$ and edge set $E(G) \cup E(H)$.

\begin{result}\label{union}
    If $G$ and $H$ satisfy \Cref{conj}, then so does $G \sqcup H$.
\end{result}

\begin{proof}
    It can be checked that it is enough to prove the result for when $G$ and $H$ are both path graphs. 
    Let $m, n \geq 1$. 
    Consider the path graph $P_{m +n}$ whose vertex set is $[m +n]$ and consider the subgraph obtained by deleting the edge $\{m, m + 1\}$. 
    This subgraph is clearly the same as $P_m \sqcup P_n$. 
    It is easy to see that any forcing set in $P_m \sqcup P_n$ is also forcing in $P_{m + n}$ and hence $z(P_m \sqcup P_n; i) \leq z(P_{m + n}; i)$ for all $i \geq 1$.
\end{proof}

\section{Bounds via graph operations}\label{results_sec}

In this section, we consider various operations on graphs such as deleting or adding vertices and edges, and examine how they affect the number of zero forcing sets. 
We then use these results to show that various classes of graphs satisfy \Cref{conj}.

\subsection{Leaves and hanging cycles}

A \emph{leaf} in a graph is a vertex of degree $1$. 
We have the following result for graphs that have a leaf.

\begin{lemma}\label{leafdel}
    Let $G$ be a graph and $x$ be a leaf that is adjacent to a vertex $v$. 
    For any $i \geq 1$, we have $$z'(G; i) \geq z'(G \setminus \{x\}; i) + z'(G \setminus \{x, v\}; i - 1).$$
    In particular, if $G \setminus \{x\}$ and $G \setminus \{x,v\}$ both satisfy \Cref{conj}, then so does $G$.
\end{lemma}

\begin{proof}
We prove the result by describing an injection $$\iota : Z'(G \setminus \{x\}; i) \cup Z'(G \setminus \{x, v\}; i - 1) \hookrightarrow Z'(G, i).$$

Let $S \in Z'(G \setminus \{x\}; i)$. 
If $v \in S$, we set $\iota(S) = S \setminus \{v\} \cup \{x\}$. 
The fact that $\iota(S) \in Z'(G; i)$ follows because the only vertex that $x$ can force is $v$. 
If $v \notin S$, we set $\iota(S) = S$, which is non-forcing in $G$ as well.

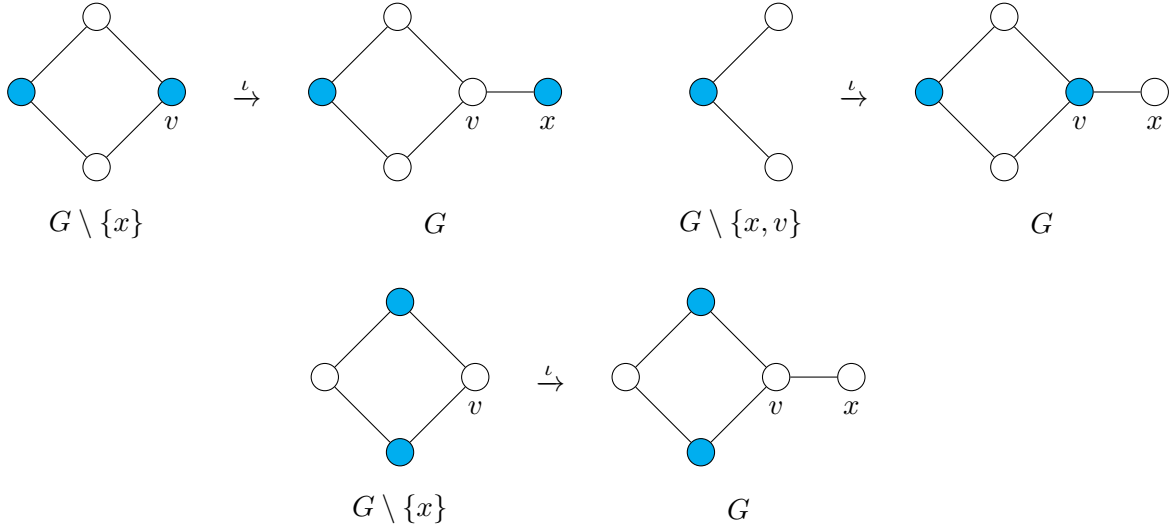
\begin{figure}[H]
    \centering
    \begin{tikzpicture}
        \node[circle, draw = black, fill = cyan, label = below: {$v$}] (v) at (3, 2) {};
        \node[circle, draw = black, fill = cyan] (a) at (1, 2) {};
        \node[circle, draw = black] (b) at (2, 1) {};
        \node[circle, draw = black] (c) at (2, 3) {};

        \draw (a) -- (c) -- (v) -- (b) -- (a);

        \node at (2, 0.25) {$G \setminus \{x\}$};

        \node at (4, 2) {$\xrightarrow{\iota}$};

        \node[circle, draw = black, label = below: {$v$}] (v1) at (3 + 4, 2) {};
        \node[circle, draw = black, fill = cyan] (a1) at (1 + 4, 2) {};
        \node[circle, draw = black] (b1) at (2 + 4, 1) {};
        \node[circle, draw = black] (c1) at (2 + 4, 3) {};
        \node[circle, draw = black, fill = cyan, label = below: {$x$}] (x) at (4 + 4, 2) {};

        \draw (a1) -- (c1) -- (v1) -- (b1) -- (a1);
        \draw (v1) -- (x);

        \node at (6.5, 0.25) {$G$};
    \end{tikzpicture}
    \hfill
    \begin{tikzpicture}
        \node[circle, draw = black, fill = cyan] (a) at (1, 2) {};
        \node[circle, draw = black] (b) at (2, 1) {};
        \node[circle, draw = black] (c) at (2, 3) {};

        \draw (a) -- (c);
        \draw (b) -- (a);

        \node at (1.5, 0.25) {$G \setminus \{x, v\}$};

        \node at (3, 2) {$\xrightarrow{\iota}$};

        \node[circle, draw = black, label = below: {$v$}, fill = cyan] (v1) at (3 + 3, 2) {};
        \node[circle, draw = black, fill = cyan] (a1) at (1 + 3, 2) {};
        \node[circle, draw = black] (b1) at (2 + 3, 1) {};
        \node[circle, draw = black] (c1) at (2 + 3, 3) {};
        \node[circle, draw = black, label = below: {$x$}] (x) at (4 + 3, 2) {};

        \draw (a1) -- (c1) -- (v1) -- (b1) -- (a1);
        \draw (v1) -- (x);

        \node at (5.5, 0.25) {$G$};
    \end{tikzpicture}
    
    \vspace{0.5cm}
    
    \begin{tikzpicture}
        \node[circle, draw = black, label = below: {$v$}] (v) at (3, 2) {};
        \node[circle, draw = black] (a) at (1, 2) {};
        \node[circle, draw = black, fill = cyan] (b) at (2, 1) {};
        \node[circle, draw = black, fill = cyan] (c) at (2, 3) {};

        \draw (a) -- (c) -- (v) -- (b) -- (a);

        \node at (2, 0.25) {$G \setminus \{x\}$};

        \node at (4, 2) {$\xrightarrow{\iota}$};

        \node[circle, draw = black, label = below: {$v$}] (v1) at (3 + 4, 2) {};
        \node[circle, draw = black] (a1) at (1 + 4, 2) {};
        \node[circle, draw = black, fill = cyan] (b1) at (2 + 4, 1) {};
        \node[circle, draw = black, fill = cyan] (c1) at (2 + 4, 3) {};
        \node[circle, draw = black, label = below: {$x$}] (x) at (4 + 4, 2) {};

        \draw (a1) -- (c1) -- (v1) -- (b1) -- (a1);
        \draw (v1) -- (x);

        \node at (6.5, 0.25) {$G$};
    \end{tikzpicture}
    \caption{Instances of the injection $\iota$.}
\end{figure}

Let $S \in Z'(G \setminus \{x, v\}; i - 1)$. 
We set $\iota(S) = S \cup \{v\}$. 
Since $x \notin \iota(S)$, the vertex $v$ can, at most, be used to color $x$. 
But in this case, all neighbors of $v$ (apart from $x$) have to already be colored. 
Hence, the vertex $v$ cannot be used to color more vertices in $G \setminus \{x, v\}$ than $S$ can. 
This shows that $\iota(S) \in Z'(G; i)$.

It can be verified that $\iota$ is an injection. 
Now suppose that $G$ has $n$ vertices and $G \setminus \{x\}$ and $G \setminus \{x, v\}$ satisfy \Cref{conj}. 
We have for any $i \geq 1$,
\begin{align*}
    z'(G; i) &\geq z'(G \setminus \{x\}; i) + z'(G \setminus \{x, v\}; i - 1)\\
    &\geq z'(P_{n - 1}; i) + z'(P_{n - 2}; i - 1)\\
    &= \binom{(n - 1) - i - 1}{i} + \binom{(n - 2) - (i - 1) - 1}{i - 1} = \binom{n - i - 1}{i} = z'(P_n; i).
\end{align*}
Hence, $G$ satisfies \Cref{conj}.
\end{proof}

Recall that a \emph{tree} is a connected graph that contains no cycles. 
Since any tree has a leaf, induction on the number of vertices in a tree (along with \Cref{union}) gives us the following result.

\begin{corollary}\label{treeresult}
    For any tree $T$ on $n$ vertices, $z(T;i) \leq z(P_n;i)$ for all $i \geq 1$.
\end{corollary}

The above corollary is a strengthening of \cite[Theorem 1.6]{sam24} and also proves a conjecture in Sam Spiro's list of open problems \cite[Conjecture 3.7]{p2solve}.

We now consider deletion of an edge in certain special cycles in a graph. 
We say that a graph has a \emph{hanging cycle} on an edge $e = \{v, w\}$ if it has a cycle $v, a_1, \ldots, a_k, w, v$ such that for all $i \in [k]$, the vertex $a_i$ has degree $2$. 
There is no constraint on the degree of the vertices $v$ and $w$.

\begin{lemma}\label{hangingcycle}
    Let $G$ be a graph with a hanging cycle on an edge $e$. 
    Let $G'$ be the graph obtained from $G$ by removing the edge $e$. 
    We have $z(G;i)\leq z(G';i)$ for all $i \geq 1$. 
    In particular, if $G'$ satisfies \Cref{conj}, so does $G$.
\end{lemma}

\begin{proof}
    We will in fact show that $Z(G; i) \subseteq Z(G'; i)$ for all $i \geq 1$. 
    Let $S$ be a forcing set in $G$. 
    We have to show that $S$ is forcing in $G'$ as well.
    
    Suppose $e = \{v, w\}$ and $v, a_1, a_2, \ldots, a_k, w$ is the cycle hanging from the edge $e$. 
    If $S$ can color $G$ without using the edge $e$ to force a vertex, then it is a forcing set in $G'$ as well. 
    Suppose we use the edge $e$ in the coloring. 
    Note that the edge $e$ only plays a role in the forcing if one of $v$ or $w$ is used to force the other.

    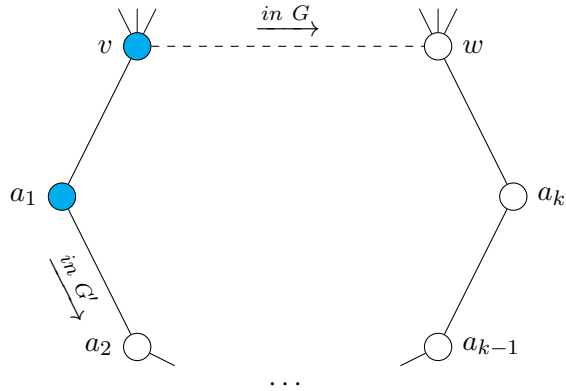
\begin{figure}[H]
        \centering
        \begin{tikzpicture}
            \node[circle, draw = black, fill = cyan, label = left: {$v$}] (v) at (2, 6) {};
            \node[circle, draw = black, fill = cyan, label = left: {$a_1$}] (a1) at (1, 4) {};
            \node[circle, draw = black, label = left: {$a_2$}] (a2) at (2, 2) {};
            \draw (v) -- (a1);
            \draw (a1) -- (a2) node [midway, sloped, below] {$\xrightarrow{in\ G'}$};
            \draw (a2) -- (2.5, 1.75);

            \node at (4, 1.5) {$\cdots$};

            \node[circle, draw = black, label = right: {$a_{k - 1}$}] (ak1) at (6, 2) {};
            \node[circle, draw = black, label = right: {$a_{k}$}] (ak) at (7, 4) {};
            \node[circle, draw = black, label = right: {$w$}] (w) at (6, 6) {};
            \draw (ak1) -- (ak) -- (w);
            \draw (ak1) -- (5.5, 1.75);

            \draw[dashed] (v) -- (w) node [midway, above] {$\xrightarrow{in\ G}$};
            \draw[ultra thin] (v) -- (1.75, 6.5);
            \draw[ultra thin] (v) -- (2, 6.5);
            \draw[ultra thin] (v) -- (2.25, 6.5);
            \draw[ultra thin] (w) -- (4+1.75, 6.5);
            \draw[ultra thin] (w) -- (4+2, 6.5);
            \draw[ultra thin] (w) -- (4+2.25, 6.5);
        \end{tikzpicture}
        \caption{Forcing sets remain forcing after deleting $e$.}
    \end{figure}
    
    Without loss of generality, suppose $v$ and all its neighbors in $G$ apart from $w$ are colored at some point in the forcing. 
    Then we can use $e$ to color $w$ in $G$. 
    However, in $G'$, we can use $a_1$ (which is colored since it is a neighbor of $v$) to color $a_2$, which colors $a_3$, and so on. 
    Finally, coloring along the cycle, we can color $w$. 
    Hence, $S$ is a forcing set in $G'$ as well.
\end{proof}

The lemmas above allow us to prove that certain planar graphs satisfy \Cref{conj}. 
Recall that a planar graph is one that can be drawn on the plane with no edges crossing.

\begin{definition}\label{OP_def}
    An \emph{outerplanar graph} is a planar graph which can be drawn on the plane in such a way that all the vertices lie on the outer (unbounded) face (see \cite{outerplanar} for more details).
\end{definition}

For example, the graph in \Cref{opgraph} is outerplanar whereas the complete graph $K_4$, although planar, is not outerplanar. 
In particular, all trees are outerplanar.

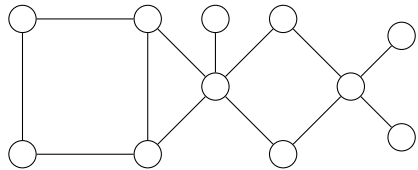
\begin{figure}[H]
    \centering
    \begin{tikzpicture}[scale=0.9]
        \node [circle, draw = black] (a) at (0.15, 0) {};
        \node [circle, draw = black] (b) at (2, 0) {};
        \node [circle, draw = black] (c) at (4, 0) {};
        \node [circle, draw = black] (d) at (3, 1) {};
        \node [circle, draw = black] (e) at (0.15, 2) {};
        \node [circle, draw = black] (f) at (2, 2) {};
        \node [circle, draw = black] (g) at (4, 2) {};
        \node [circle, draw = black] (h) at (5, 1) {};
        \node [circle, draw = black] (i) at (3, 2) {};
        \node [circle, draw = black] (j) at (5.75, 1.75) {};
        \node [circle, draw = black] (k) at (5.75, 0.25) {};

        \draw (d) -- (f) -- (e) -- (a) -- (b) -- (d) -- (c) -- (h) -- (g) -- (d) -- (i);
        \draw (b) -- (f);
        \draw (j) -- (h) -- (k);
    \end{tikzpicture}
    \caption{An outerplanar graph.}
    \label{opgraph}
\end{figure}

To show that all outerplanar graphs satisfy \Cref{conj}, we use the fact that any outerplanar graph has either a leaf or a hanging cycle. 
To prove this, we first need a few definitions.

A \emph{cut vertex} of a connected graph is a vertex whose deletion results in a disconnected graph. 
A graph is said to be \emph{biconnected} (or $2$-connected) if it is connected and has no cut vertices. 
The maximal biconnected subgraphs of a connected graph are called its \emph{blocks}. 
Cut vertices are precisely those vertices that lie in more than one block.

Biconnected outerplanar graphs are relatively simple to describe. 
Such graphs with at most two vertices are trivial to characterize. 
Biconnected outerplanar graphs with at least three vertices are precisely those graphs obtained from cycle graphs by adding any number of non-crossing edges between the vertices (see \cite{outerplanar}).

\begin{figure}[H]
    \centering
    \begin{tikzpicture}[scale=0.9]
    \foreach \a in {1,...,9}
    {
        \node[circle, draw = black] (\a) at ({\a*40}:2){};
    }
    \draw (1)--(2)--(3)--(4)--(5)--(6)--(7)--(8)--(9)--(1);
    \draw (1)--(3);
    \draw (9)--(4)--(7);
    \end{tikzpicture}
    \caption{A biconnected outerplanar graph.}
\end{figure}
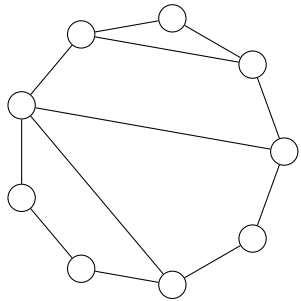

We will need another definition before proving that outerplanar graphs satisfy \Cref{conj}. 
The \emph{block-cut tree} of a connected graph $G$ is a tree whose vertices correspond to the blocks and cut vertices of $G$. 
There is an edge between a block and a cut vertex if the cut vertex lies in the block. 
It can be verified that the graph defined does indeed form a tree.

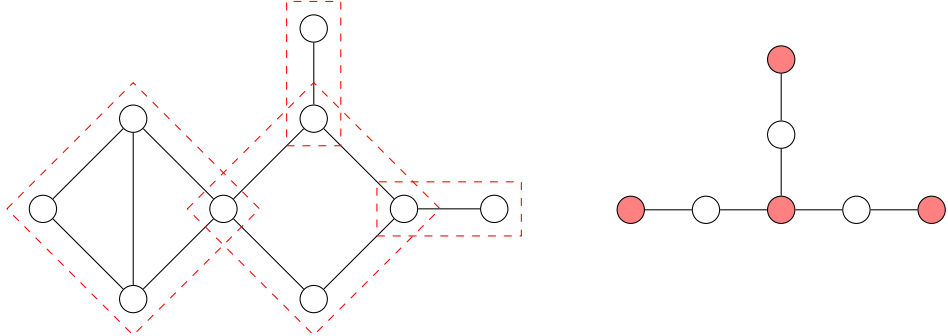
\begin{figure}[H]
    \centering
    \begin{tikzpicture}[scale = 1.2]
        \node [circle, draw = black] (a) at (0, 2) {};
        \node [circle, draw = black] (b) at (1, 3) {};
        \node [circle, draw = black] (c1) at (2, 2) {};
        \node [circle, draw = black] (d) at (1, 1) {};

        \node [circle, draw = black] (c2) at (3, 3) {};
        \node [circle, draw = black] (e) at (3, 4) {};
        \node [circle, draw = black] (c3) at (4, 2) {};
        \node [circle, draw = black] (g) at (3, 1) {};

        \node [circle, draw = black] (f) at (5, 2) {};

        \draw (a)--(b)--(c1)--(c2)--(c3)--(g)--(c1)--(d)--(a);
        \draw (d) -- (b);
        \draw (c2) -- (e);
        \draw (c3) -- (f);

        \draw [red, dashed, thin] (-0.4, 2) -- (1, 3.4) -- (2.4, 2) -- (1, 0.6) -- cycle;
        \draw [red, dashed, thin] (-0.4 + 2, 2) -- (1 + 2, 3.4) -- (2.4 + 2, 2) -- (1 + 2, 0.6) -- cycle;
        \draw [red, dashed, thin] (2.7, 2.7) rectangle (3.3, 4.3);
        \draw [red, dashed, thin] (3.7, 1.7) rectangle (5.3, 2.3);
    \end{tikzpicture}
    \hspace{1cm}
    \begin{tikzpicture}
        \node [circle, draw = black, fill = red!50] (a) at (0, 0) {};
        \node [circle, draw = black] (b) at (1, 0) {};
        \node [circle, draw = black, fill = red!50] (c) at (2, 0) {};
        \node [circle, draw = black] (d) at (3, 0) {};
        \node [circle, draw = black, fill = red!50] (e) at (4, 0) {};

        \node [circle, draw = black] (f) at (2, 1) {};
        \node [circle, draw = black, fill = red!50] (g) at (2, 2) {};

        \draw (a) -- (b) -- (c) -- (d) -- (e);
        \draw (c) -- (f) -- (g);

        \node at (2, -1.55) {};
    \end{tikzpicture}
    \caption{A graph and its block-cut tree with blocks specified in red.}
\end{figure}

\begin{theorem}\label{opresult}
    All outerplanar graphs satisfy \Cref{conj}.
\end{theorem}

\begin{proof}
    Let $O$ be an outerplanar graph. 
    We use induction on the sum of the number of vertices and edges of $O$. 
    By \Cref{union}, we can assume that $O$ is connected. 
    If $O$ has a leaf, we use \Cref{leafdel} and the induction hypothesis. 
    Note that any subgraph of an outerplanar graph is outerplanar.
    
    Next suppose that $O$ has no leaves. 
    If $O$ is biconnected, then it is obtained from a cycle graph by adding non-crossing edges between the vertices. 
    Hence, it is clear that $O$ will have a hanging cycle and we use \Cref{hangingcycle} and the induction hypothesis.

    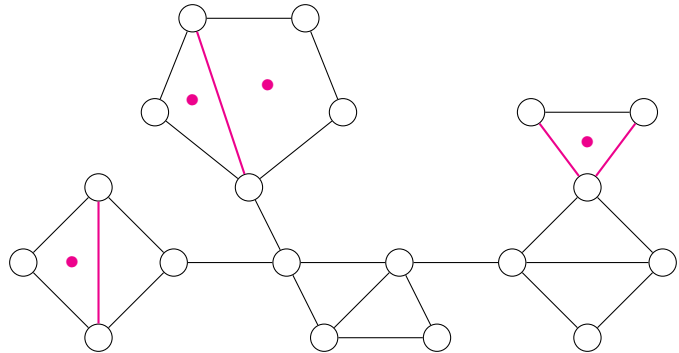
\begin{figure}[H]
        \centering
        \begin{tikzpicture}
            \node [circle, draw = black] (a) at (0, 1) {};
            \node [circle, draw = black] (b) at (1, 2) {};
            \node [circle, draw = black] (c) at (2, 1) {};
            \node [circle, draw = black] (d) at (1, 0) {};
            \node [circle, draw = black] (e) at (3.5, 1) {};
            \node [circle, draw = black] (f) at (3, 2) {};
            \node [circle, draw = black] (g) at (1.75, 3) {};
            \node [circle, draw = black] (h) at (2.25, 4.25) {};
            \node [circle, draw = black] (i) at (3.75, 4.25) {};
            \node [circle, draw = black] (j) at (4.25, 3) {};
            \node [circle, draw = black] (k) at (5, 1) {};
            \node [circle, draw = black] (l) at (4, 0) {};
            \node [circle, draw = black] (m) at (5.5, 0) {};
            \node [circle, draw = black] (n) at (7 - 0.5, 1) {};
            \node [circle, draw = black] (o) at (8 - 0.5, 2) {};
            \node [circle, draw = black] (p) at (9 - 0.5, 1) {};
            \node [circle, draw = black] (q) at (8 - 0.5, 0) {};
            \node [circle, draw = black] (r) at (6.75, 3) {};
            \node [circle, draw = black] (s) at (8.25, 3) {};

            \node at (0.65, 1) {\color{magenta} $\bullet$};
            \node at (2.25, 3.15) {\color{magenta} $\bullet$};
            \node at (3.25, 3.35) {\color{magenta} $\bullet$};
            \node at (7.5, 2.6) {\color{magenta} $\bullet$};

            \draw (a)--(b)--(c)--(d)--(a);
            \draw [thick,magenta] (b)--(d);
            \draw (c)--(e)--(f)--(g)--(h)--(i)--(j)--(f);
            \draw [thick,magenta] (h)--(f);
            \draw (e)--(k)--(m)--(l)--(e);
            \draw (l)--(k);
            \draw (k)--(n)--(o)--(p)--(q)--(n)--(p);
            \draw [thick, magenta] (s)--(o)--(r);
            \draw (r)--(s);
        \end{tikzpicture}
        \caption{Outerplanar graph with hanging cycles in `leaf blocks' specified (along with edge(s) they hang on). Note that there are other hanging cycles as well.}
        \label{fig:enter-label}
    \end{figure}
    
    If $O$ is not biconnected, consider a block $B$ of $O$ that has only one cut vertex $v$ of $O$ (this is a block that corresponds to a leaf in the block-cut tree of $O$). 
    Note that the subgraph $B$ is a biconnected outerplanar graph and hence is obtained from a cycle graph by adding non-crossing edges between the vertices. 
    We can always find a hanging cycle in the subgraph $B$ that is either hanging by an edge incident to $v$ or one that does not contain the vertex $v$. 
    Such a hanging cycle is a hanging cycle of $O$ as well and we use \Cref{hangingcycle} and the induction hypothesis.
\end{proof}

The idea used in the proof of the above theorem in fact gives us a slightly more general result. 
Let $G$ and $H$ be graphs with disjoint vertex sets. 
Let $v \in V(G)$ and $w \in V(W)$. 
The wedge of $H$ with $G$ at $v$ (using $w$) is the graph obtained from $G \sqcup H$ by identifying the vertices $v$ and $w$. 
When the vertex $v \in V(G)$ or $w \in V(H)$ used in the identification does not matter, it is not mentioned.

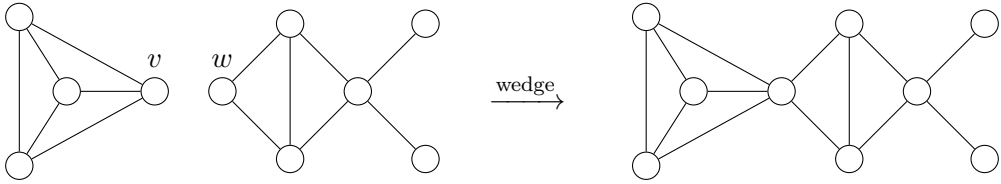
\begin{figure}[H]
    \centering
    \begin{tikzpicture}[scale=0.9]
        \node [circle, draw = black, label ={$v$}] (1) at (-1, 2) {};
        \node [circle, draw = black] (2) at (-2.3, 2) {};
        \node [circle, draw = black] (3) at (-3, 3.1) {};
        \node [circle, draw = black] (4) at (-3, 0.9) {};
        \draw (1)--(4)--(3)--(1);
        \draw (4)--(2)--(3);
        \draw (2)--(1);
        
        \node [circle, draw = black, label = {$w$}] (a) at (0, 2) {};
        \node [circle, draw = black] (b) at (1, 3) {};
        \node [circle, draw = black] (c1) at (2, 2) {};
        \node [circle, draw = black] (d) at (1, 1) {};

        \node [circle, draw = black] (c2) at (3, 3) {};
        \node [circle, draw = black] (g) at (3, 1) {};

        \draw (a)--(b)--(c1)--(c2);
        \draw (g)--(c1)--(d)--(a);
        \draw (d) -- (b);

        \node at (4.5, 2) {$\xrightarrow{\text{wedge}}$};
    \end{tikzpicture}
    \hspace{0.5cm}
    \begin{tikzpicture}[scale=0.9]
        \node [circle, draw = black] (2) at (-2.3 + 1, 2) {};
        \node [circle, draw = black] (3) at (-3 + 1, 3.1) {};
        \node [circle, draw = black] (4) at (-3 + 1, 0.9) {};
        
        \node [circle, draw = black] (a) at (0, 2) {};
        \draw (a)--(4)--(3)--(a);
        \draw (4)--(2)--(3);
        \draw (2)--(a);
        \node [circle, draw = black] (b) at (1, 3) {};
        \node [circle, draw = black] (c1) at (2, 2) {};
        \node [circle, draw = black] (d) at (1, 1) {};

        \node [circle, draw = black] (c2) at (3, 3) {};
        \node [circle, draw = black] (g) at (3, 1) {};

        \draw (a)--(b)--(c1)--(c2);
        \draw (g)--(c1)--(d)--(a);
        \draw (d) -- (b);
    \end{tikzpicture}
    \caption{Wedge of two graphs.}
    \label{wedge}
\end{figure}

\begin{corollary}\label{wedgeOP}
    Let $G$ be a graph with a vertex $v$ such that both $G$ and $G \setminus \{v\}$ satisfy \Cref{conj}. 
    Any graph obtained from $G$ by wedging an outerplanar graph at $v$ also satisfies \Cref{conj}.
\end{corollary}

\begin{proof}
    Let $O$ be the outerplanar graph we wedge at $v$. 
    We use induction on the sum of the number of vertices and edges of $O$. 
    If $O$ has a leaf (other than $v$), then we use \Cref{leafdel} and apply the induction hypothesis. 
    Note that even if such a vertex is adjacent to $v$, deleting $v$ gives a disjoint union of $G \setminus \{v\}$ and an outerplanar graph and \Cref{union} shows that this union also satisfies \Cref{conj}. 
    Next, suppose that $O$ has no such leaf. 
    Using similar ideas as in the proof of \Cref{opresult}, it can be shown that $O$ must have a hanging cycle that is also a hanging cycle in the wedge. 
    In this case, we apply \Cref{hangingcycle}.
\end{proof}

For example, the above corollary can be used to show that the wedged graph in \Cref{wedge} satisfies \Cref{conj}.

The following is an immediate corollary to \Cref{hangingcycle}.

\begin{corollary}\label{wedgecyclepath}
    Let $G$ be a graph and $v \in V(G)$. 
    Let $G'$ be the graph obtained by wedging the cycle graph $C_n$ with $G$ at $v$. 
    Let $G''$ be the graph obtained by wedging the path $P_n$ with $G$ at $v$ using a leaf vertex of the path. 
    Then $z(G'; i) \leq z(G''; i)$ for all $i \geq 1$.
\end{corollary}

In some sense, the graph $G''$ in the above corollary is more `path-like' than $G'$. 
We will see another result (\Cref{hangingpaths}) where a more `path-like' graph has more forcing sets of each cardinality than one that is less `path-like'.

\subsection{Cones and simplicial vertices}

For any graph $G$, the \emph{cone} over $G$, denoted by $c(G)$, is the graph obtained by adding a new vertex $v$ that is adjacent to all vertices of $G$. 
Here $v$ is called the cone vertex.

\begin{lemma}
    Let $G$ be a graph and $c(G)$ be the cone over $G$ with cone vertex $v$.
    \begin{enumerate}
        \item If $G$ has no isolated vertices and satisfies \Cref{conj}, then so does $c(G)$.

        \item If $G$ has exactly one isolated vertex $x$ and $G \setminus \{x\}$ satisfies \Cref{conj}, then so does $c(G)$.
        
        \item If $G$ has at least $2$ isolated vertices, then $c(G)$ satisfies \Cref{conj}.
    \end{enumerate}
\end{lemma}

\begin{proof}
\begin{enumerate}
    \item To prove the first statement, we note that if $S \in Z'(G; i)$ for some $i \geq 1$, then $S \in Z'(c(G); i)$ and $S \cup \{v\} \in Z'(c(G); i + 1)$. 
    This is because, assuming $v$ has been colored at some point, it can only force a vertex $w$ in $G$ if all other vertices in $G$ are already colored. 
    But this would mean that $S \in Z(G; i)$ since $w$ could be forced by any of its neighbors in $G$.

    \begin{figure}[H]
        \centering
        \begin{tikzpicture}
            \node[circle, draw = black, fill = cyan, label = above: {$v$}] (v) at (3, 4.5) {};
            \node[circle, draw = black, label = right: {$w$}] (w) at (5.5, 2) {};
            \node[circle, draw = black, fill = cyan] (a) at (1, 2) {};
            \node[circle, draw = black, fill = cyan] (b) at (1, 3) {};
            \node[circle, draw = black, fill = cyan] (c) at (2, 2) {};
            \node[circle, draw = black, fill = cyan] (d) at (3.5, 2) {};

            \draw (a) -- (c) -- (b) -- (a);
            
            \draw (d) -- (w) node [midway, below] {$\xrightarrow{in\ G}$};
            
            \draw (a) -- (v) -- (b);
            \draw (c) -- (v) -- (d);

            \draw (v) -- (w) node [sloped, midway, above] {$\xrightarrow{in\ c(G)}$};
        \end{tikzpicture}
        \caption{Vertex forced by $v$ in $c(G)$ can be forced in $G$ as well.}
    \end{figure}
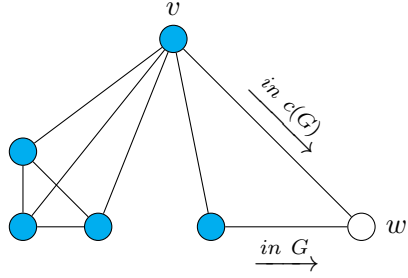
    
    Let $n$ be the number of vertices in $G$. 
    Since $G$ satisfies \Cref{conj}, this shows that for any $i \geq 1$,
    \begin{align*}
        z'(c(G); i) &\geq z'(G; i) + z'(G; i - 1)\\
        &\geq z'(P_n; i) + z'(P_n; i - 1)\\
        &= \binom{n - i - 1}{i} + \binom{n - (i - 1) - 1}{i - 1}\\
        &\geq \binom{n - i - 1}{i} + \binom{n - i - 1}{i - 1} = \binom{(n + 1) - i - 1}{i} = z'(P_{n + 1}; i).
    \end{align*}
    Hence, $c(G)$ also satisfies \Cref{conj}.

    \item The second statement follows from \Cref{leafdel} and the first statement. 
    Note that $x$ becomes a leaf in $c(G)$.

    \item The third statement follows from \Cref{fort} since any two isolated vertices in $G$ form a fort in $c(G)$.
\end{enumerate}
\end{proof}

In particular, the above lemma shows, by induction on the number of vertices, that all \emph{threshold graphs} satisfy \Cref{conj}.

\begin{definition}\label{thresh_def}
    \emph{Threshold graphs} are defined recursively as follows:
    \begin{enumerate}
        \item The graph consisting of just one vertex is a threshold graph.
        \item If $G$ is a threshold graph, so is $c(G)$.
        \item If $G$ is a threshold graph, so is the graph obtained by adding an isolated vertex to $G$.
    \end{enumerate}
\end{definition}

\begin{figure}[ht]
    \centering
    \begin{tikzpicture}[scale=0.5]
    \node (2) [circle,draw=black] at (0,0) {};
    \node (4) at (0.75,-2) {};
    \node at (1.5,0) {$\rightarrow$};
    \end{tikzpicture}
    \begin{tikzpicture}[scale=0.5]
        \node (2) [circle,draw=black] at (0,0) {};
        \node (3) [circle,draw=black] at (2,0) {};
        \node (4) at (0.75,-2) {};
        \node at (3.5,0) {$\rightarrow$};
    \end{tikzpicture}
    \begin{tikzpicture}[scale=0.5]
        \node (2) [circle,draw=black] at (0,0) {};
        \node (3) [circle,draw=black] at (2,0) {};
        \node (1) [circle,draw=black] at (1,2) {};
        \node (4) at (0.75,-2) {};
        \node at (3.5,0) {$\rightarrow$};
        \draw (3)--(1)--(2);
    \end{tikzpicture}
    \begin{tikzpicture}[scale=0.5]
        \node (2) [circle,draw=black] at (0,0) {};
        \node (3) [circle,draw=black] at (2,0) {};
        \node (1) [circle,draw=black] at (1,2) {};
        \draw (3)--(1)--(2);
        \node at (3.5,0) {$\rightarrow$};
        \node (4) [circle,draw=black] at (1,-2) {};
        \draw (1)--(4);
        \draw (2)--(4);
        \draw (3)--(4);
    \end{tikzpicture}
    \begin{tikzpicture}[scale=0.5]
        \node (2) [circle,draw=black] at (0,0) {};
        \node (3) [circle,draw=black] at (2,0) {};
        \node (1) [circle,draw=black] at (1,2) {};
        \draw (3)--(1)--(2);
        \node (4) [circle,draw=black] at (1,-2) {};
        \node (5) [circle,draw=black] at (3.5,0) {};
        \draw (1)--(4);
        \draw (2)--(4);
        \draw (3)--(4);
    \end{tikzpicture}
    \caption{Construction of a threshold graph on $5$ vertices.}
    \label{tgconst1}
\end{figure}
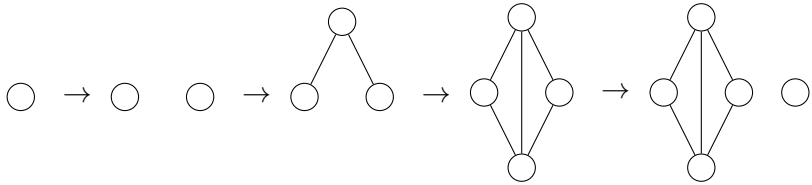

We note that the zero forcing polynomial of threshold graphs has been studied in \cite[Section 4.2]{boyer19}.

\begin{corollary}\label{TOPW}
    Let $G$ be a threshold graph and $v_1, v_2, \ldots, v_k$ be vertices of $G$. 
    Let $O_1, O_2, \ldots, $ $O_k$ be outerplanar graphs. 
    The graph obtained by wedging the graph $O_i$ with $G$ at $v_i$ for all $i \in [k]$ satisfies \Cref{conj}.
\end{corollary}

\begin{proof}
    Note that deleting a vertex from a threshold (respectively, outerplanar) graph results in a threshold (respectively, outerplanar) graph. 
    The result now follows using induction on $k$ and \Cref{wedgeOP}.
\end{proof}

Similar to hanging cycles, we now consider another operation where we delete certain edges of a graph. 
A vertex $v$ in a graph is called a \emph{simplicial vertex} if all the vertices adjacent to $v$ are adjacent to each other, i.e., its neighborhood $N(v)$ is a clique.

\begin{lemma}
    Let $G$ be a graph with a simplicial vertex $v$. 
    Let $G'$ be a graph obtained from $G$ by deleting some of the edges between vertices of $N(v)$. 
    For all $i \geq 1$, we have $z(G; i) \leq z(G'; i)$.
\end{lemma}

\begin{proof}
    We show that $Z(G; i) \subseteq Z(G'; i)$ for all $i \geq 1$. 
    Let $S$ be a forcing set in $G$. 
    Suppose the edge $e = \{w_1, w_2\}$ where $w_1, w_2 \in N(v)$ plays a role in the coloring of $G$. 
    This happens only if, without loss of generality, at some point in the coloring, $w_1$ and all its neighbors apart from $w_2$ are colored and $w_1$ is used to color $w_2$. 
    Note that in this case, none of the other edges between vertices of $N(v)$ played a role in the coloring up to that point since $w_2$ was uncolored.

    \begin{figure}[H]
        \centering
        \begin{tikzpicture}
            \node[circle, draw = black, fill = cyan, label = right: {$v$}] (v) at (6, 4) {};
            \node[circle, draw = black, label = below: {$w_2$}] (w2) at (1, 4) {};
            \node[circle, draw = black, fill = cyan, label = right: {$w_1$}] (w1) at (3, 6) {};
            \node[circle, draw = black, fill = cyan] (w) at (3, 2) {};

            \draw[dashed] (w1) -- (w2) node [midway, sloped, above] {$\xleftarrow{in\ G}$};
            \draw[dashed] (w2) -- (w) -- (w1);
            \draw (w) -- (v) -- (w1);
            \draw (v) -- (w2) node [above, pos = 0.35] {$\xleftarrow{in\ G'}$};

            \draw[ultra thin] (w1) -- (2.75, 6.5);
            \draw[ultra thin] (w1) -- (3, 6.5);
            \draw[ultra thin] (w1) -- (3.25, 6.5);

            \draw[ultra thin] (w2) -- (0.5, 3.75);
            \draw[ultra thin] (w2) -- (0.5, 4);
            \draw[ultra thin] (w2) -- (0.5, 4.25);

            \draw[ultra thin] (w) -- (2.75, 1.5);
            \draw[ultra thin] (w) -- (3, 1.5);
            \draw[ultra thin] (w) -- (3.25, 1.5);
        \end{tikzpicture}
        \caption{Forcing sets in $G$ are forcing in $G'$.}
    \end{figure}
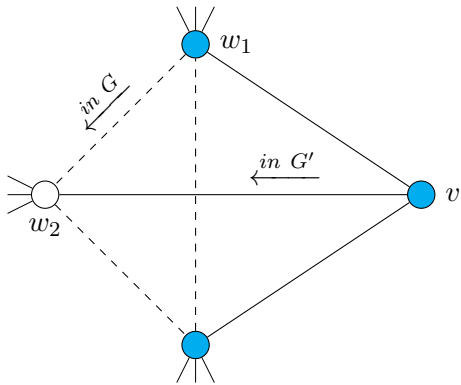

    Instead of using $w_1$ to color $w_2$, we note that $v$ is colored, and all the vertices in $N(v)$ apart from $w_2$ are colored. 
    Hence, in $G'$, $v$ can be used to color $w_2$. 
    This shows that $S$ is forcing in $G'$ as well.
\end{proof}

\begin{corollary}
    Let $G$ be a graph that satisfies \Cref{conj} and $v \in V(G)$. 
    The graph obtained by adding all edges between the vertices of $N(v)$ (thus making $v$ a simplicial vertex) also satisfies \Cref{conj}.
\end{corollary}

\subsection{Hanging paths}

A \emph{hanging path} at a vertex $v$ of a graph $G$ is a path $a_1, a_2, \ldots, a_k$ in $G$ such that
\begin{itemize}
    \item $a_1$ is adjacent to $v$,
    \item for all $i \in [k - 1]$, the vertex $a_i$ has degree $2$, and
    \item the vertex $a_k$ has degree $1$.
\end{itemize}

We now show that `combining' the paths hanging from a vertex increases the number of forcing sets of each cardinality. 
One way to interpret this result is that graphs that are more `path-like' have more zero forcing sets of each cardinality.

\begin{proposition}\label{hangingpaths}
    Suppose that a graph $G$ has two paths $A$ and $B$ hanging at a vertex $v$. 
    Let $G'$ be the graph obtained from $G$ by removing $B$ and attaching it at the end of $A$ (see \Cref{fig:hanging paths}). 
    For all $i \geq 1$, we have $z(G;i) \leq z(G';i)$.
\end{proposition}

\begin{proof}
    Suppose that the path $A$ has vertices $a_1, a_2, \ldots, a_k$ and the path $B$ has vertices $b_1, b_2, \ldots, b_m$ with the vertices $a_1, b_1$ adjacent to $v$. 
    The graph $G'$ is obtained from $G$ by deleting the edge $\{v, b_1\}$ and adding the edge $\{a_k, b_1\}$. 
    We use $C$ to denote the path $a_1, a_2, \ldots, a_k, b_1, b_2, \ldots, b_m$ hanging from $v$ in $G'$.

    \begin{figure}[ht]
        \centering
        \begin{tikzpicture}[scale=.8]
            \node[circle, draw = black, label = left: {$a_k$}] (a4) at (-0.4, -1) {};
            \node[circle, draw = black, label = left: {$a_{k - 1}$}] (a3) at (0.2, 0.5) {};
            \node[circle, draw = black, label = left: {$a_2$}] (a2) at (0.8, 2) {};
            \node[circle, draw = black, label = left: {$a_1$}] (a1) at (1.4, 3.5) {};
            \node[circle, draw = black, label = left: {$v$}] (w) at (2, 5) {};
            \node[circle, draw = black, label = right: {$b_1$}] (b1) at (2.6, 3.5) {};
            \node[circle, draw = black, label = right: {$b_2$}] (b2) at (3.2, 2) {};
            \node[circle, draw = black, label = right: {$b_{m - 1}$}] (b3) at (3.8, 0.5) {};
            \node[circle, draw = black, label = right: {$b_m$}] (b4) at (4.4, -1) {};

            \draw (a4) -- (a3);
            \draw (a3) -- (a2) node[midway, sloped, fill = white] {\tiny $\cdots$};
            \draw (a2) -- (a1) -- (w) -- (b1) -- (b2);
            \draw (b2) -- (b3) node[midway, sloped, fill = white] {\tiny $\cdots$};
            \draw (b3) -- (b4);

            \draw[ultra thin] (w) -- (1.75, 5.5);
            \draw[ultra thin] (w) -- (2, 5.5);
            \draw[ultra thin] (w) -- (2.25, 5.5);

            \draw [gray, thin, decorate,decoration={brace,amplitude=10pt, mirror,raise=6ex}] (1.4, 3.6) -- (-0.4, -1.3);
            \node at (-1.35, 1.75) {$A$};
            \draw [gray, thin, decorate,decoration={brace,amplitude=10pt,raise=6ex}] (2.6, 3.6) -- (4.4, -1.3);
            \node at (5.35, 1.75) {$B$};

            \node[circle, draw = black, label = left: {$v$}] (w') at (7 + 3, 5) {};
            \node[circle, draw = black, label = left: {$a_1$}] (a1') at (7 + 3, 4) {};
            \node[circle, draw = black, label = left: {$a_2$}] (a2') at (7 + 3, 3) {};
            \node[circle, draw = black, label = left: {$a_k$}] (a3') at (7 + 3, 1.5) {};
            \node[circle, draw = black, label = right: {$b_1$}] (b1') at (7 + 3, 0.5) {};
            \node[circle, draw = black, label = right: {$b_2$}] (b2') at (7 + 3, -0.5) {};
            \node[circle, draw = black, label = right: {$b_m$}] (b3') at (7 + 3, -2) {};
            
            \draw (w') -- (a1') -- (a2');
            \draw (a2') -- (a3') node[midway, sloped, fill = white] {\tiny $\dots$};
            \draw (a3') -- (b1') -- (b2');
            \draw (b2') -- (b3') node[midway, sloped, fill = white] {\tiny $\dots$};

            \draw[ultra thin] (w') -- (6.75 + 3, 5.5);
            \draw[ultra thin] (w') -- (7 + 3, 5.5);
            \draw[ultra thin] (w') -- (7.25 + 3, 5.5);

            \draw [gray, thin, decorate,decoration={brace,amplitude=10pt,raise=5ex}] (10, 4.2) -- (10, -2.2);
            \node at (11.75, 1) {$C$};

            \node at (2, -3.25) {$G$};
            \node at (10, -3.25) {$G'$};
        \end{tikzpicture}
        \caption{The graphs $G$ and $G'$ of \Cref{hangingpaths}.}
        \label{fig:hanging paths}
    \end{figure}
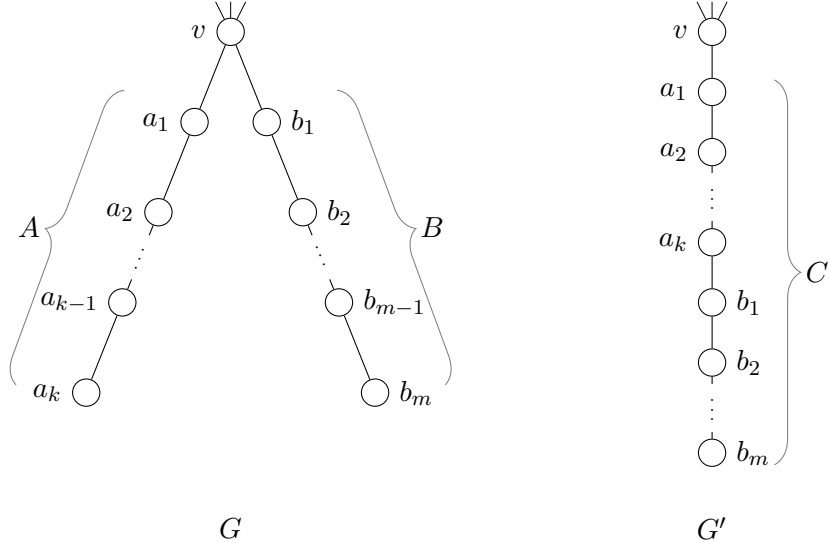

    Fix $i \geq 1$. 
    We prove the result by describing an injection $\iota: Z(G; i) \hookrightarrow Z(G'; i)$. 
    Suppose $S \in Z(G; i)$.

   \noindent \textbf{\underline{Case 1:}} We first assume that $v \in S$. 
    In this case, we set $\iota(S) = S$ (note that $V(G) = V(G')$). 
    We have to show that $S \in Z(G'; i)$.
    
    First, we consider the case when, in $G$, all neighbors of $v$ not in $A$ or $B$ can be colored without using $v$ to force any coloring. 
    Then the same is true in $G'$ as well. 
    However, in $G'$, this implies that all vertices in $C$ can now be colored since all neighbors of $v$ apart from $a_1$ are colored. 
    Hence, $v$ forces $a_1$, which in turn forces $a_2$, and so on. 
    This shows that $S \in Z(G'; i)$.

    Next, suppose that in $G$, the vertex $v$ has to be used to color one of its neighbors that is not in $A$ or $B$. 
    This means that the vertices $a_1$ and $b_1$ can be colored in $G$ without using the vertex $v$. 
    This, in turn, means that $a_1$ was colored using the vertices in $S \cap A$ and $b_1$ was colored using the vertices in $S \cap B$. 
    To show that $S \in Z(G'; i)$, we just have to show that $S \cap C$ can force the vertex $a_1$ in the path graph $C$. 
    If $a_1 \in S$, this is trivially true. 
    If $a_1 \notin S$, then either a consecutive pair of vertices of $A$ or the vertex $a_k$ belongs to $S$ (see \Cref{path_result}). 
    If a consecutive pair of vertices of $A$ is in $S$, then $S \cap C$ forces $a_1$. 
    Assuming that there is no such pair, we have $a_k \in S$. 
    Now, since $b_1$ can be colored in $B$ using $S \cap B$, the same method can be used to color $b_1$ in $S \cap C$. 
    Hence, both $b_1$ and $a_k$ can be colored, and we get a consecutive pair of colored vertices in $C$, which can be used to force $a_1$.

    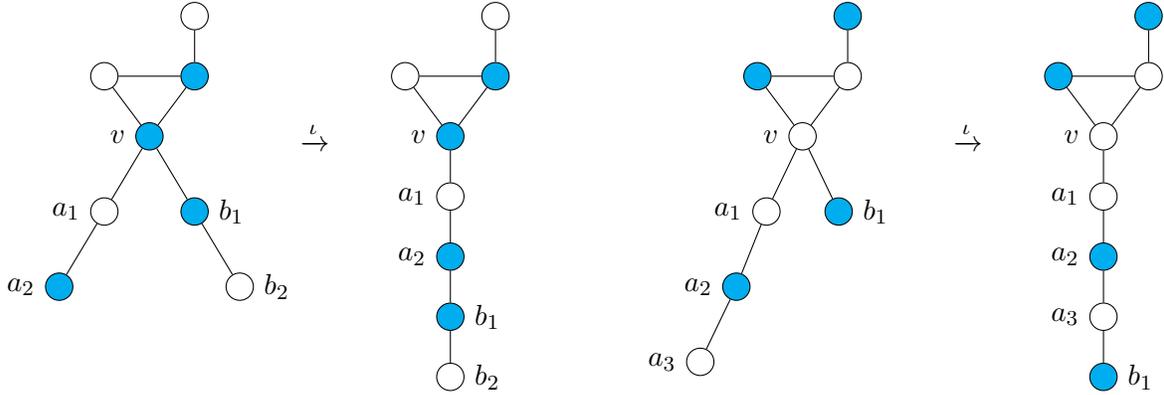
\begin{figure}[H]
        \centering
        \begin{tikzpicture}[scale = 0.8]
            \node[circle, draw = black, fill = cyan, label = left: {$a_2$}] (a2) at (0.5, 2.5) {};
            \node[circle, draw = black, label = left: {$a_1$}] (a1) at (1.25, 3.75) {};
            \node[circle, draw = black, fill = cyan, label = left: {$v$}] (w) at (2, 5) {};
            \node[circle, draw = black, fill = cyan, label = right: {$b_1$}] (b1) at (2.75, 3.75) {};
            \node[circle, draw = black, label = right: {$b_2$}] (b2) at (3.5, 2.5) {};
            \node[circle, draw = black] (a) at (1.25, 6) {};
            \node[circle, draw = black, fill = cyan] (b) at (2.75, 6) {};
            \node[circle, draw = black] (c) at (2.75, 7) {};

            \draw (a2) -- (a1) -- (w) -- (a) -- (b) -- (w) -- (b1) -- (b2);
            \draw (b) -- (c);

            \node at (4.75, 5) {$\xrightarrow{\iota}$};

            \node[circle, draw = black, fill = cyan, label = left: {$v$}] (w') at (2 + 5, 5) {};
            \node[circle, draw = black] (a') at (1.25 + 5, 6) {};
            \node[circle, draw = black, fill = cyan] (b') at (2.75 + 5, 6) {};
            \node[circle, draw = black] (c') at (2.75 + 5, 7) {};
            \draw (c') -- (b') -- (a') -- (w') -- (b');
            \node[circle, draw = black, label = left: {$a_1$}] (a1') at (7, 4) {};
            \node[circle, draw = black, fill = cyan, label = left: {$a_2$}] (a2') at (7, 3) {};
            \node[circle, draw = black, fill = cyan, label = right: {$b_1$}] (b1') at (7, 2) {};
            \node[circle, draw = black, label = right: {$b_2$}] (b2') at (7, 1) {};
            \draw (w') -- (a1') -- (a2') -- (b1') -- (b2');
        \end{tikzpicture}
        \hfill
        \begin{tikzpicture}[scale = 0.8]
            \node[circle, draw = black, label = left: {$a_3$}] (a3) at (0.3, 1.25) {};
            \node[circle, draw = black, label = left: {$a_2$}, fill = cyan] (a2) at (0.9, 2.5) {};
            \node[circle, draw = black, label = left: {$a_1$}] (a1) at (1.4, 3.75) {};
            \node[circle, draw = black, label = left: {$v$}] (w) at (2, 5) {};
            \node[circle, draw = black, label = right: {$b_1$}, fill = cyan] (b1) at (2.6, 3.75) {};
            \node[circle, draw = black, fill = cyan] (a) at (1.25, 6) {};
            \node[circle, draw = black] (b) at (2.75, 6) {};
            \node[circle, draw = black, fill = cyan] (c) at (2.75, 7) {};

            \draw (a3) -- (a2) -- (a1) -- (w) -- (a) -- (b) -- (w) -- (b1);
            \draw (b) -- (c);

            \node at (4.75, 5) {$\xrightarrow{\iota}$};

            \node[circle, draw = black, label = left: {$v$}] (w') at (2 + 5, 5) {};
            \node[circle, draw = black, fill = cyan] (a') at (1.25 + 5, 6) {};
            \node[circle, draw = black] (b') at (2.75 + 5, 6) {};
            \node[circle, draw = black, fill = cyan] (c') at (2.75 + 5, 7) {};
            \draw (c') -- (b') -- (a') -- (w') -- (b');
            \node[circle, draw = black, label = left: {$a_1$}] (a1') at (7, 4) {};
            \node[circle, draw = black, fill = cyan, label = left: {$a_2$}] (a2') at (7, 3) {};
            \node[circle, draw = black, label = left: {$a_3$}] (b1') at (7, 2) {};
            \node[circle, draw = black, label = right: {$b_1$}, fill = cyan] (b2') at (7, 1) {};
            \draw (w') -- (a1') -- (a2') -- (b1') -- (b2');
        \end{tikzpicture}
        \caption{Instances of $\iota$ corresponding to \textbf{Case 1} and \textbf{Case 2}.}
    \end{figure}

    \noindent\textbf{\underline{Case 2:}} Suppose that $v \notin S$ and $v$ can be forced in $G$ using neither $a_1$ nor $b_1$. 
    In this case as well, we set $\iota(S) = S$ and hence we have to show that $S \in Z(G'; i)$.

    This case is very similar to the previous case. 
    Note that any coloring done in $G \setminus (A \cup B)$ can be done in $G' \setminus C$ as well. 
    Our assumption says that $S$ can be used to color $v$ in $G \setminus (A \cup B)$ and hence in $G' \setminus C$. 
    Now we can assume $v$ is colored and this brings us back to the previous case.

    \noindent\textbf{\underline{Case 3:}} Suppose that $v \notin S$ and $v$ can only be forced in $G$ using $a_1$ or $b_1$. 
    This means that $S$ restricted to the path in $G$ induced by the vertices $A \cup \{v\}$ or the path induced by $B \cup \{v\}$ is a forcing set (where $v$ is not colored). 
    Hence, one of the following must hold:
    \begin{itemize}
        \item There are two consecutive vertices in $S \cap A$ or $a_k \in S$.
        \item There are two consecutive vertices in $S \cap B$ or $b_m \in S$.
    \end{itemize}
    
    In all cases \emph{apart} from when
    \begin{itemize}
        \item there are no consecutive vertices in $S \cap A$ or $S \cap B$, and
        \item $a_k \in S$ and $b_1, b_m \notin S$,
    \end{itemize}
    we define $\iota(S) = S$. 
    It can be checked that in all cases apart from the one mentioned above, $S$ restricted to the path $C \cup \{v\}$ is a forcing set. 
    This gives us that $S \in Z(G'; i)$.
    
    In the case mentioned above, we define $\iota(S)$ by `reversing the coloring on $A \cup \{v\}$'. 
    We identify any subset $X$ of $V(G)$ (or $V(G')$) with the choice function
    \begin{equation*}
        c_X(w) = 
        \begin{cases}
            1, \text{ if }w \in X\\
            0, \text{ if }w \notin X.
        \end{cases}
    \end{equation*}
    Now, setting $a_0 := v$, we define $\iota(S)$ by
    \begin{itemize}
        \item $c_{\iota(S)}(a_l) = c_S(a_{k - l})$ for all $0 \leq l \leq k$, and
        
        \item $c_{\iota(S)}(w) = c_S(w)$ for all $w \in V(G \setminus (A \cup \{v\}))$.
    \end{itemize}

    \begin{figure}[H]
        \centering
        \begin{tikzpicture}[scale = 0.8]
            \node[circle, draw = black, fill = cyan, label = left: {$a_2$}] (a2) at (0.5, 2.5) {};
            \node[circle, draw = black, label = left: {$a_1$}, fill = cyan] (a1) at (1.25, 3.75) {};
            \node[circle, draw = black, label = left: {$v$}] (w) at (2, 5) {};
            \node[circle, draw = black, label = right: {$b_1$}] (b1) at (2.75, 3.75) {};
            \node[circle, draw = black, label = right: {$b_2$}] (b2) at (3.5, 2.5) {};
            \node[circle, draw = black] (a) at (1.25, 6) {};
            \node[circle, draw = black] (b) at (2.75, 6) {};
            \node[circle, draw = black, fill = cyan] (c) at (2.75, 7) {};

            \draw (a2) -- (a1) -- (w) -- (a) -- (b) -- (w) -- (b1) -- (b2);
            \draw (b) -- (c);

            \node at (4.75, 5) {$\xrightarrow{\iota}$};

            \node[circle, draw = black, label = left: {$v$}] (w') at (2 + 5, 5) {};
            \node[circle, draw = black] (a') at (1.25 + 5, 6) {};
            \node[circle, draw = black] (b') at (2.75 + 5, 6) {};
            \node[circle, draw = black, fill = cyan] (c') at (2.75 + 5, 7) {};
            \draw (c') -- (b') -- (a') -- (w') -- (b');
            \node[circle, draw = black, label = left: {$a_1$}, fill = cyan] (a1') at (7, 4) {};
            \node[circle, draw = black, fill = cyan, label = left: {$a_2$}] (a2') at (7, 3) {};
            \node[circle, draw = black, label = right: {$b_1$}] (b1') at (7, 2) {};
            \node[circle, draw = black, label = right: {$b_2$}] (b2') at (7, 1) {};
            \draw (w') -- (a1') -- (a2') -- (b1') -- (b2');
        \end{tikzpicture}
        \hfill
        \begin{tikzpicture}[scale = 0.8]
            \node[circle, draw = black, fill = cyan, label = left: {$a_3$}] (a3) at (0.3, 1.25) {};
            \node[circle, draw = black, label = left: {$a_2$}] (a2) at (0.9, 2.5) {};
            \node[circle, draw = black, label = left: {$a_1$}, fill = cyan] (a1) at (1.4, 3.75) {};
            \node[circle, draw = black, label = left: {$v$}] (w) at (2, 5) {};
            \node[circle, draw = black, label = right: {$b_1$}] (b1) at (2.6, 3.75) {};
            \node[circle, draw = black, fill = cyan] (a) at (1.25, 6) {};
            \node[circle, draw = black] (b) at (2.75, 6) {};
            \node[circle, draw = black] (c) at (2.75, 7) {};

            \draw (a3) -- (a2) -- (a1) -- (w) -- (a) -- (b) -- (w) -- (b1);
            \draw (b) -- (c);

            \node at (4.75, 5) {$\xrightarrow{\iota}$};

            \node[circle, draw = black, label = left: {$v$}, fill = cyan] (w') at (2 + 5, 5) {};
            \node[circle, draw = black, fill = cyan] (a') at (1.25 + 5, 6) {};
            \node[circle, draw = black] (b') at (2.75 + 5, 6) {};
            \node[circle, draw = black] (c') at (2.75 + 5, 7) {};
            \draw (c') -- (b') -- (a') -- (w') -- (b');
            \node[circle, draw = black, label = left: {$a_1$}] (a1') at (7, 4) {};
            \node[circle, draw = black, fill = cyan, label = left: {$a_2$}] (a2') at (7, 3) {};
            \node[circle, draw = black, label = left: {$a_3$}] (b1') at (7, 2) {};
            \node[circle, draw = black, label = right: {$b_1$}] (b2') at (7, 1) {};
            \draw (w') -- (a1') -- (a2') -- (b1') -- (b2');
        \end{tikzpicture}
        \caption{Instances of the injection $\iota$ corresponding to \textbf{Case 3}.}
    \end{figure}
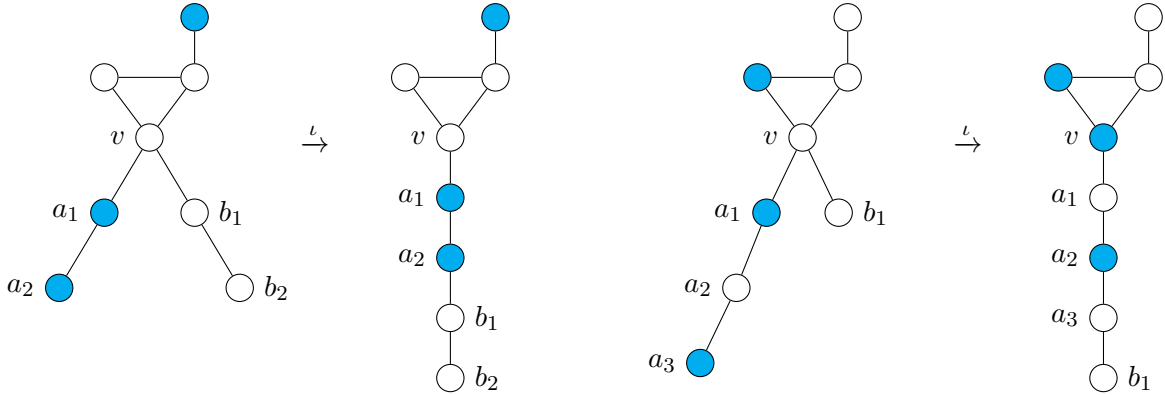

    We have to show that $\iota(S) \in Z(G'; i)$. 
    Note that for $S$ in the graph $G$, the vertex $v$ has an uncolored neighbor $b_1$. 
    Also, by the conditions on $S$, $b_1$ can only be colored using $v$. 
    This means that, at least after $v$ has been colored, all the neighbors of $v$ apart from $b_1$ can be colored without using the vertex $v$ to force any coloring.
    
    Looking at $\iota(S)$ in $G'$, note that since $v \in \iota(S)$ and $\iota(S) \setminus C = S \setminus (A \cup B)$, all the neighbors of $v$ in $G'$ apart from $a_1$ can be colored without using $v$ to force any coloring. 
    This gives us that $v$ can be used to force $a_1$, which forces $a_2$, and so on. 
    Hence, $\iota(S) \in Z(G'; i)$. 
    It can also be checked that in this case, $\iota(S) \notin Z(G; i)$ (neither $a_1$ nor $b_1$ can be colored). 
    Hence, $\iota$ is an injection, as required.
\end{proof}

Let $T$ be a tree with $n$ vertices. 
A pleasing consequence of the above proposition is that repeated application of its proof gives an injection $Z(T; i) \hookrightarrow Z(P_n; i)$ for all $i \geq 1$. 
This gives an alternate proof of \Cref{treeresult}. 
In fact, we can prove the slightly stronger result given below, which was suggested by Sam Spiro \cite{samcommunication}.

\begin{corollary}
    Let $T$ be a tree on $n$ vertices that is \emph{not} the path graph. 
    For any $1 \leq i < \frac{n}{2}$, we have $z(T; i) < z(P_n; i)$.
\end{corollary}

\begin{proof}
    It can be checked that path graphs are the only graphs that have forcing sets of size $1$. 
    Hence, the result is true for $i = 1$ and we have to prove it for $2 \leq i \leq \lfloor\frac{n - 1}{2}\rfloor$. 
    Repeatedly applying \Cref{hangingpaths}, we see that it is enough to prove the result for the case when $T$ has a vertex $v$ with two hanging paths $A$ and $B$ such that $T \setminus (A \cup B)$ is a path graph.

    To prove the result, we show that the injection $\iota: Z(T; i) \hookrightarrow Z(P_n; i)$ (defined in the proof of \Cref{hangingpaths}) is not surjective. 
    We label the vertices of $A$ and $B$ just as before and label the vertices in the path graph $T \setminus (A \cup B \cup \{v\})$ as $c_1, c_2, \ldots, c_p$ where $c_p$ is adjacent to $v$. 
    Hence, the path graph $P_n$ we consider is the one with vertices labeled $c_1, c_2, \ldots, c_p, v, a_1, a_2, \ldots, a_k, b_1, b_2, \ldots, b_m$.

    Consider any $S \subset V(P_n)$ that satisfies the following conditions: 
    \begin{itemize}
        \item Both $v, a_1 \in S$.
        \item When $S$ is restricted to the path induced by $\{b_1, b_2, \ldots, b_m\}$, we have a non-forcing set. In particular, $b_1 \notin S$. 
        \item When $S$ is restricted to the path induced by $\{c_1, c_2, \ldots, c_p\}$, we have a non-forcing set. In particular, $c_p \notin S$. 
    \end{itemize}
    Since $v, a_1 \in S$, it is clear that any such $S$ is forcing in $P_n$. 
    Note that if we consider $S$ (or $S$ with the coloring on $A \cup \{v\}$ `reversed') as a subset of $T$, then neither $b_1$ nor $c_p$ can be colored and hence $S$ is non-forcing in $T$.

    \begin{figure}[ht]
        \centering
        \begin{tikzpicture}[xscale = 0.6,scale = 0.9]
            \node [circle, draw = black, label = left: {$a_3$}, fill = cyan] (a3) at (0, 1) {};
            \node [circle, draw = black, label = left: {$a_2$}] (a2) at (1, 2) {};
            \node [circle, draw = black, label = left: {$a_1$}, fill = cyan] (a1) at (2, 3) {};

            \node [circle, draw = black, label = left: {$v$}, fill = cyan] (v) at (3, 4) {};

            \node [circle, draw = black, label = right: {$b_4$}] (b4) at (7 - 0.4, 0) {};
            \node [circle, draw = black, label = right: {$b_3$}, fill = cyan] (b3) at (6 -0.3, 1) {};
            \node [circle, draw = black, label = right: {$b_2$}] (b2) at (5 -0.2, 2) {};
            \node [circle, draw = black, label = right: {$b_1$}] (b1) at (4 -0.1, 3) {};

            \node [circle, draw = black, label = right: {$c_3$}] (c3) at (3, 5) {};
            \node [circle, draw = black, label = right: {$c_2$}, fill = cyan] (c2) at (3, 6) {};
            \node [circle, draw = black, label = right: {$c_1$}] (c1) at (3, 7) {};

            \draw (a3)--(a2)--(a1)--(v)--(b1)--(b2)--(b3)--(b4);
            \draw (v)--(c3)--(c2)--(c1);
        \end{tikzpicture}
        \caption{A non-forcing set in $T$ that is forcing in $P_n$.}
    \end{figure}
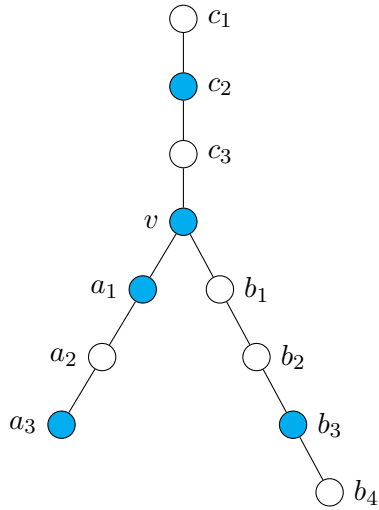

    Since the largest size of any non-forcing set in a path graph with $l$ vertices is $\lfloor\frac{l - 1}{2}\rfloor$, we can construct such an $S$ of size $i$ for all $i$ such that $2 \leq i \leq \left\lfloor\frac{m - 1}{2}\right\rfloor + \left\lfloor\frac{p - 1}{2}\right\rfloor + k + 1$. 
    Using the fact that $k, m, p \geq 1$ and that for any $r, s \geq 1$, $\lfloor\frac{r - 1}{2}\rfloor + \lfloor\frac{s - 1}{2}\rfloor + 1 \geq \lfloor\frac{r + s - 1}{2}\rfloor$, we have
    \begin{align*}
        \left\lfloor\frac{m - 1}{2}\right\rfloor + \left\lfloor\frac{p - 1}{2}\right\rfloor + k + 1 &\geq \left\lfloor\frac{m + p - 1}{2}\right\rfloor + k\\
        &\geq \left\lfloor\frac{m + p - 1}{2}\right\rfloor + \left\lfloor\frac{k}{2}\right\rfloor + 1 \geq \left\lfloor\frac{n - 1}{2}\right\rfloor.
    \end{align*} 
    This proves the required result.
\end{proof}

\section{Future directions}

Upon having \Cref{treeresult}, one would hope that \Cref{conj} can be proved by showing that any connected graph $G$ has a spanning tree $T$ such that $z(G; i) \leq z(T; i)$ for all $i \geq 1$. 
Unfortunately, this is not true. 
It can be checked that the graph in \Cref{spantree} does not have a spanning tree with this property.

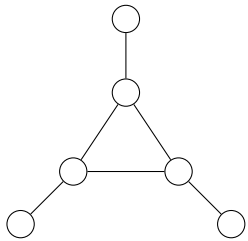
\begin{figure}[ht]
    \centering
    \begin{tikzpicture}[scale = 0.7]
        \node [circle, draw = black] (1) at (0, 0) {};
        \node [circle, draw = black] (2) at (1, 1) {};
        \node [circle, draw = black] (3) at (2, 2.5) {};
        \node [circle, draw = black] (4) at (2, 3.9) {};
        \node [circle, draw = black] (5) at (3, 1) {};
        \node [circle, draw = black] (6) at (4, 0) {};

        \draw (1)--(2)--(3)--(4);
        \draw (3)--(5)--(6);
        \draw (5)--(2);
    \end{tikzpicture}
    \caption{A graph $G$ without a spanning tree $T$ such that $[G] \leq [T]$.}
    \label{spantree}
\end{figure}

However, this does lead to asking whether, given a graph $G$, one can find a `nice' graph $H$ such that $z(G; i) \leq z(H; i)$ for all $i \geq 1$. 
Ideally, $H$ should be a graph that satisfies \Cref{conj} and has the same number of vertices as $G$.

One way to address this question, as we have done in this article, is to consider graph operations that increase the number of forcing sets of each cardinality. 
For example, some other operations that could be studied are: subdivision of an edge, product of graphs, and wedge of graphs.

Another direction is to consider the following equivalence and poset structure on graphs: 
We say that two graphs $G$ and $H$ with $n$ vertices are equivalent if $z(G; i) = z(H; i)$ for all $i \geq 1$. 
Denote the equivalence class of a graph $G$ by $[G]$. 
This equivalence has been studied in \cite[Section 5.1]{boyer19}. 
For example, they show that $[P_n] = \{P_n\}$. 
We define a poset structure on these equivalence classes by setting $[G] \leq [H]$ if $z(G; i) \leq z(H; i)$ for all $i \geq 1$. 
Note that \Cref{conj} asserts that $[P_n]$ is the maximum element of this poset.

What else can be said about this poset? 
For example, what are the graphs $G$ such that $[G] \leq [C_n]$? 
If $[P_n]$ is indeed the maximum of this poset, what can be said about the coatoms, i.e., the elements covered by $[P_n]$?

\section*{Acknowledgements}
The authors would like to thank Sam Spiro and Luyining Gan for constructive comments on the early draft of this article, which significantly enhanced its overall presentation. Krishna Menon is partially supported by a grant from the Infosys Foundation. 
This work was done during the first author's visit to the Indian Institute of Technology Bhilai, funded by SERB, India, via the SRG/2022/000314 project. 
The computer algebra system SageMath \cite{Sage} provided valuable assistance in studying examples.

\bibliographystyle{abbrv}

\end{document}